\documentclass{birkjour}
\usepackage{xcolor} 
\usepackage{array,multirow}

\newtheorem{theorem}{Theorem}[section]

\newtheorem{proposition}[theorem]{Proposition}
\newtheorem{definition}[theorem]{Definition}
\newtheorem{remark}[theorem]{Remark}
\numberwithin{equation}{section}

\begin{document}

%
%
%

%
%
%
%
%

\title[On evolutoids of regular surfaces]
 {On evolutoids of regular surfaces in Eu\-cli\-dean $3$-space}

\author[Ady Cambraia Jr]{Ady Cambraia Jr$^\star$}

\address{Departamento de Matem\'atica, Universidade Federal de Vi\c cosa, 36570-900, Brazil.}
\email{ady.cambraia@ufv.br}


\thanks{Work partially  supported  by  FAPEMIG}
\thanks{$^\star$Corresponding author}

\author{Ab\'ilio Lemos}
\address{Departamento de Matem\'atica, Universidade Federal de Vi\c cosa, 36570-900, Brazil.}
\email{abiliolemos@ufv.br}

\author{Mostafa Salarinoghabi}
\address{Departamento de Matem\'atica, Universidade Federal de Vi\c cosa, 36570-900, Brazil.}
\email{mostafa.salarinoghabi@ufv.br}



\date{March 13, 2019}

\begin{abstract}
Inspired by the concept of evolutoids of planar curves, we present the concept of evolutoids for regular surfaces as an envelope of a two-parameter family of lines in Euclidean $3$-space. We give an explicit parametrization for such evolutoids. Besides, we used some standard differential geometry techniques to study the local behavior of regular points of this object and presented some relations between the geometry of the regular surface and its evolutoid.
\end{abstract}

\maketitle

\section{Introduction}\label{intro}
Let $F:\mathbb{R}^n\times\mathbb{R}^k\to\mathbb{R}^d$ be a $k$-parameters family of a smooth map $f:\mathbb R^n\to\mathbb R^d$. The envelope, or the discriminant, of the family  $F$ is the set 
\[\mathcal{D}=\mathcal{D}_F=\left\{x\in \mathbb R^n: \text{there exists}\, t\in\mathbb R^k\,\text{with}\,F(x,t)=\dfrac{\partial F}{\partial t}(x,t)=0\right\},\]
where for $t=(u_1,\dots,u_k)\in\mathbb R^k$, the equation 
$\partial F/\partial t(x,t)=0$
means \linebreak
$\partial F/\partial u_i(x,t)=0$ for all $1\leq i \leq k$. If $x\in \mathcal{D}$ then $t$ is said to correspond to $x$. The set of $x$s corresponding to a given $t$ is called the \textit{characteristic} corresponding to that value of $t$ (for more details we refer the reader to \cite{bruce}). The evolute of a parametric plane curve is a well-known example of characteristic corresponded to a real value $t$. Precisely speaking, the evolute of a plane curve $\gamma$ is the characteristic given by the envelope of the family of normal lines of $\gamma$. It is widely known  that  the cusps in the evolute correspond to the curvature extrema of $\gamma$. The evolute also goes to infinity corresponding to inflections (zeros of curvature) of $\gamma$.

Another interesting curve associated to the plane curve $\gamma$
is its \textit{e\-vo\-lu\-toid}: that is, the characteristic obtained by the envelope of lines making a fixed angle ($\alpha\in (0,\pi/2)$) with the normal line at every point of the curve. This is the subject
of several investigations, (see for instance \cite{giblin1,jeronimo1, jeronimo2, hamman}). In \cite{giblin1} the authors considered non-oval plane curves. Also, they introduced some standard techniques of singularity theory to explain how the evolutoid changes when the angle $\alpha$ ranges between 0 and $\pi/2$.

To the best of our knowledge, there are no sensible generalizations of evolutoids to higher dimensions. Hence, it is natural to wonder how one can generalize the concept of evolutoid for surfaces in  Euclidean 3-space. The focal set of a surface in $\mathbb R^3$ is the analogue of the evolute of a plane curve. Let $S$ be a smooth and regular surface in $\mathbb R^3$ without parabolic points, and let $\phi : U\subset \mathbb R^2 \to \mathbb R^3$ be a local parametrization of $S$. The
focal set (or the evolute) of $S$ is the set $\epsilon(S)=\epsilon_1(S)\cup\epsilon_2(S)$ where
\begin{align*}
\epsilon_i(p)=\left\{p+\frac{1}{\kappa_i(p)} N(p), \, p\in S\right\}, \,\, i=1,2,
\end{align*}
here $N(p)$ denotes the normal vector of $S$ at $p$ and for $i=1,2$, $\kappa_i(p)$ is the principal curvature of $S$ at the point $p$.
Given a non-umbilic point $p \in S$, two distinct focal points
corresponding to $p$ can be observed; one is on $\epsilon_1$ and the other, on $\epsilon_2$. These focal points coincide when $p$ is an umbilic point. The focal sets and their geometry are well studied, see for example \cite{BruceGiblinFarid,BruceWilkinson,izumiyatakahashi,wilkinson}.

Similar to the case of plane curves, we generalize the concept of evolutoid for the surface $S$ as the envelope of lines making a fixed angle ($\alpha\in (0,\pi/2)$) with the normal vector, $N$, at each point of the surface. There are some questions that one may ask. How is it possible to obtain the parametrization of the evolutoid? What is the regularity condition of the evolutoid? What geometric information about the surface $S$ can be derived from its evolutoid? The aim of this paper is to deal with these questions.

We have proved in this paper that, in general the evolutoid, $\mathcal E$,  is regular at the point $p\in \mathcal E$ if and only if the corresponding point in the original surface $S$ is not a ridge point. Also, the point $p$ is a parabolic point of the evolutoid $\mathcal{E}$ if and only if its corresponding point in $S$ is a subparabolic point of $S$. 

The main results of this paper are Theorem \ref{Theo:height} where we study the contact of evolutoid with planes and Theorem \ref{Theo:Sing} where we investigate the type of sin\-gu\-la\-ri\-ties of the evolutoid.

The paper is organized as follows. In section \ref{sec:Surface}, we obtain the pa\-ra\-me\-tri\-za\-tion of the evolutoid of surfaces in $\mathbb R^3$. The local differential geometry of the evolutoid and its relations with the geometry of the surface $S$ have been studied in section \ref{local}.    
\section{Evolutoid of a regular surface in $\mathbb{R}^3$}\label{sec:Surface}

Let  $F:\mathbb{R}^n\times\mathbb{R}^k\to\mathbb{R}^d$ be a $k$-parameter family of smooth functions given by $(x_1,\dots,x_n,u_1,\dots,u_k)\mapsto F(x_1,\dots,x_n,u_1,\dots,u_k)$. As mentioned in Section \ref{intro}, the envelope of the family $F$ is the solution set, $(x_1,\dots,x_n)$, of the system 

\begin{equation}\label{aa1}
\left\{\begin{array}{ccc}
F(x_1,x_2,\dots,x_n,u_1,u_2,\dots,u_k) & = & 0 \vspace{0.3cm} \\
F_{u_i}(x_1,x_2,\dots,x_n,u_1,u_2,\dots,u_k)& = & 0
\end{array}\right.
\end{equation}
 where $(x_1, x_2, \dots,x_n)$ is characteristic corresponding to the parameter   $(u_1,  \\ u_2,  \dots ,u_k)$. Also $F_{u_i}$ denotes the partial derivative $\partial F/\partial u_i,$ for $i=1,\dots,k$. Consequently, the envelope of the family $F$ exists when the rank of the coefficients matrix is less than $k+1$.

 A point $p$ on a regular surface $S$ is called a {\it parabolic point} if the Gaussian curvature of $S$, denoted by $K$, vanishes at $p$. Also, $p$ is called an {\it umbilic point} if the normal curvatures at $p$ in all directions are equal. Let $S$ be a regular surface without parabolic and umbilical points. Suppose that $\phi:U\to S$ is a parametrization of $S$, such that the coordinate curves are lines of curvature. The parametrized surfaces 
$$\epsilon_i(u,v)=\phi(u,v) + \dfrac{1}{\kappa_i}N(u,v), \quad i=1,2,$$ where $\kappa_i$'s  are the principal curvatures of $S$, are called  focal surfaces of $\phi(U)$. These surfaces can be obtained via envelope of a two-parameter family of normal lines to $S$, i.e.,
$$F(X,u,v)= X-\phi(u,v)-rN(u,v),$$
at $p=\phi(u,v)$, where $r:U\to\mathbb{R}$ is a real value function.
More precisely, we have
\begin{equation}\label{aa2}
\left\{\begin{array}{cccccc}
F_u & = & -\phi_u - rN_u & = & -(1+ra_{11})\phi_u-ra_{21}\phi_v & = 0, \\
F_v & = & -\phi_v - rN_v & = & -ra_{12}\phi_u-(1+ra_{22})\phi_v & =0.
\end{array}\right.
\end{equation}
The coefficient matrix of the system \eqref{aa2} has rank less than $2$ if 
\begin{align*} 
(1+ra_{11})(1+ra_{22})-r^2a_{12}a_{21}=0 & \Leftrightarrow \\ 
1+ r(a_{11}+a_{22})+r^2(a_{11}a_{22}-a_{12}a_{21})=0 & \Leftrightarrow \\ 1-2rH+r^2K=0&,
\end{align*}
where $H$ is the mean curvature of $S$. Thus, $r=1/\kappa_1(p)$ or $r=1/\kappa_2(p).$
 We now proceed to generalize these concepts. For fixed real values $\alpha$ and $\beta$,
consider the vector 
$$v_{\alpha,\beta}=\cos\alpha\cos\beta\cdot \dfrac{\phi_u}{||\phi_u||}+ \cos\alpha\sin\beta\cdot \dfrac{\phi_v}{||\phi_v||}+ \sin\alpha\cdot N,$$
 making a constant angle with the tangent plane of the surface $S$ at the point $p=\phi(u,v) \in S$ which is not neither a parabolic nor an umbilic point. 
%

\begin{definition}
Under the above assumptions,
the evolutoid of $S$ is the envelope of family of lines passing through the point $p$ with direction $v_{\alpha,\beta}$, i.e., of the family 
\begin{equation}\label{Q}
\begin{array}{cccl}
Q: & \mathbb{R}^3\times U &  \longrightarrow  & \mathbb{R}^3 \\
 & (\bar{X},u,v) & \longmapsto & Q(\bar{X},u,v)=\bar{X}-\phi(u,v)-rv_{\alpha,\beta},
\end{array}
\end{equation}
where $r:U\to\mathbb{R}$ is a real value function $($see Figure \ref{fig:fig1}$).$ 
\end{definition}

\begin{figure}[!htb]
	\includegraphics[clip, scale=0.35]{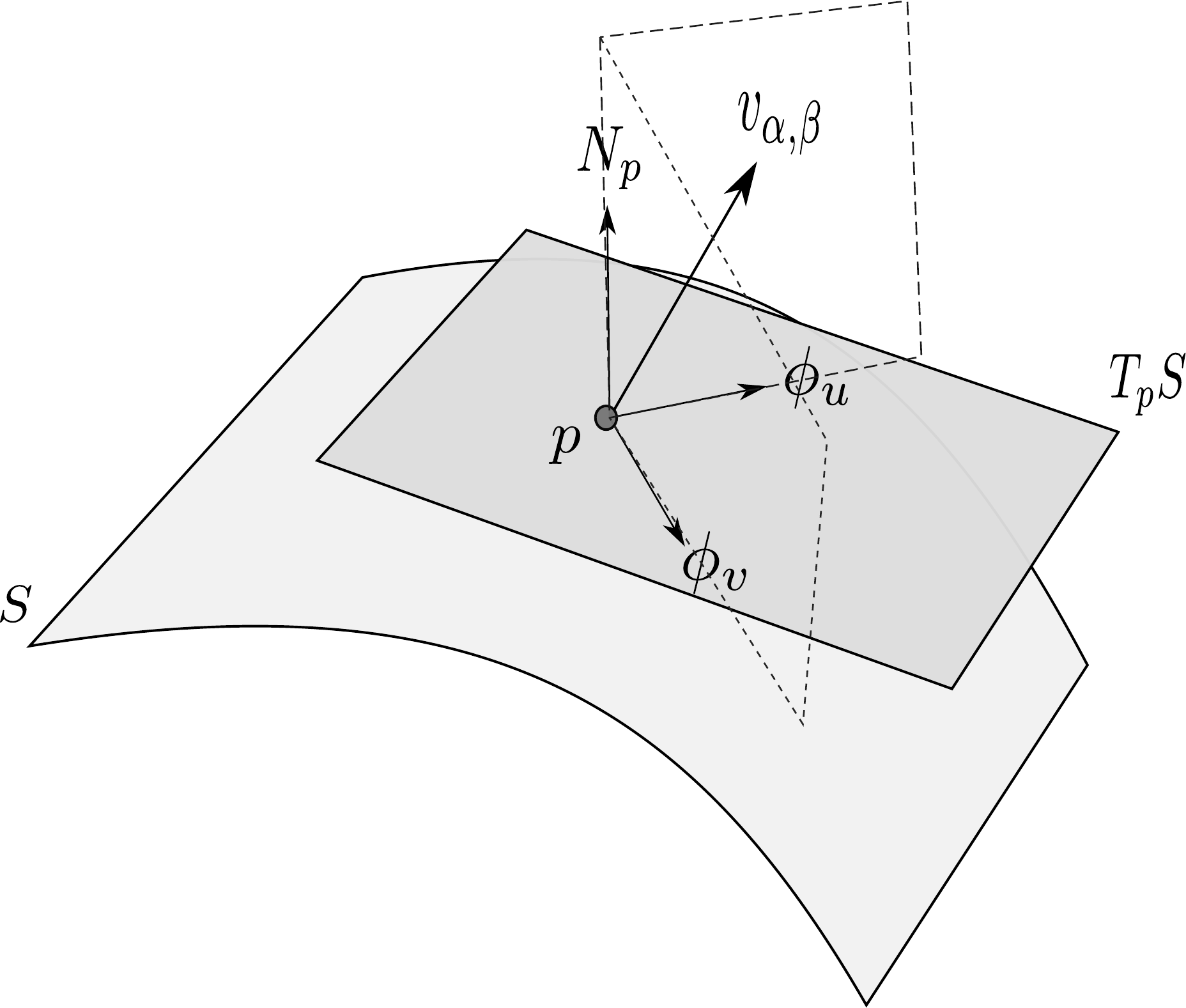}
	\caption{Lines passing through the point $p$ with direction $v_{\alpha,\beta}$.}
	\label{fig:fig1}
\end{figure}
The derivatives of  the family $Q$,  given in \eqref{Q}, are
\begin{align*}
 Q_u &=b_{11}\phi_u + b_{12}\phi_v +b_{13}N,\\
 Q_v & =b_{21}\phi_u+b_{22}\phi_v+b_{23}N,
 \end{align*}
 where $b_{ij}$ are as follow
\begin{center}
\begin{tabular}{ ll } 
 $b_{11} = -1 - r\left(\dfrac{\cos(\alpha)\sin(\beta)E_v}{2E\sqrt{G}}- \sin(\alpha)\kappa_1\right)$ & $b_{12}=\dfrac{r\cos(\alpha)\cos(\beta)E_v}{2G\sqrt{E}}$ \vspace{0.3cm}\\ 
 $b_{13}=- \dfrac{r\cos(\alpha)\cos(\beta)e}{\sqrt{E}}$ & $b_{21}=\dfrac{r\cos(\alpha)\sin(\beta)G_u}{2E\sqrt{G}}$ \vspace{0.3cm}\\ 
 $b_{22}=-1-r\left(\dfrac{\cos(\alpha)\cos(\beta)G_u}{2G\sqrt{E}}-\sin(\alpha)\kappa_2\right)$ & $b_{23}=-\dfrac{r\cos(\alpha)\sin(\beta)g}{\sqrt{G}},$
\end{tabular}
\end{center}

\noindent here $E, G$ and $e,g$ are the coefficients of the first and second fundamental forms of $S$ at the point $p$, respectively. Besides, $\kappa_1=e/E$ and $\kappa_2=g/G$ are the principal curvatures of $S$ at the point $p$.

The condition for a point $p$ to belong to the envelope of the family $Q$ is that the rank of its Jacobian matrix $(b_{ij})_{2\times3}$, denoted by $J_{\alpha,\beta}$, should be less than 2.
In the following proposition we show that this happens only for two values of $\beta$.

\begin{proposition}\label{prop:beta}
By the above assumptions the envelope of the family of lines through the point $p$ at the direction $v_{\alpha,\beta}$, given in \eqref{Q}, has solution if and only if either $$\beta=0 \ \textrm{and} \  r = -\dfrac{2\sqrt{E}G}{\cos(\alpha)G_u-2\sin(\alpha)\sqrt{E}G\kappa_2}$$ or $$\beta=\dfrac{\pi}{2} \ \textrm{and} \ r = -\dfrac{2\sqrt{E}G}{\cos(\alpha)E_v-2\sin(\alpha)\sqrt{G}E\kappa_1}.$$
\end{proposition}
\begin{proof}
	
	Let $\alpha\in(0,\pi/2)$. The rank of the matrix $J_{\alpha,\beta}=\left(b_{ij}\right)$ is less than 2 if there exists a non-zero $\lambda\in\mathbb R$ such that
	\begin{align*}
	b_{1i} = \lambda b_{2i} \quad \text{for}\quad i=1,2,3. 
	\end{align*}
	From the equation $b_{13}=\lambda b_{23}$ we obtain 
	\[\lambda = \dfrac{g \sqrt{E}  \tan(\beta)}{e \sqrt{G}},\]
	which is not zero if and only if $\tan(\beta)\ne 0$. In this case,
	by substituting $\lambda$ in the equation $b_{12} = \lambda b_{22}$ we can find $r$. It is easy to proof that, after replacing $\lambda$ and $r$ in the last equation, $b_{11} = \lambda b_{21}$, the resulting equation does not have non-zero solution for $\beta$. This means that the only possibilities for the matrix $J_{\alpha,\beta}$ to have rank one, are when one of the rows of $J_{\alpha,\beta}$ is zero. Precisely speaking, these happens when 
	\begin{align*}
\beta = 0 \quad \text{and} \quad r = -\dfrac{2\sqrt{E}G}{\cos(\alpha)G_u-2\sin(\alpha)\sqrt{E}G\kappa_2}, 
	\end{align*}  
or
	\begin{align*}
\beta = \frac{\pi}{2} \quad \text{and} \quad r = -\dfrac{2\sqrt{E}G}{\cos(\alpha)E_v-2\sin(\alpha)\sqrt{G}E\kappa_1}.
\end{align*}

Notice that, generically, we can assume the denominators of $r$'s, given above, are different from zero.
\end{proof}

As the calculations for the case $\beta=\pi/2$ are similar to one of $\beta=0$, from now on, we consider $\beta=0$. 


Therefore, using the Proposition \ref{prop:beta}, the parametrization of the envelope is
\begin{align}\label{Envelope}
\mathcal E_1(u,v)=\phi(u,v)+ r\left(\dfrac{\cos(\alpha)}{\sqrt{E}}\phi_u+\sin(\alpha)N(u,v)\right).
\end{align}
For $\beta=\pi/2$, we denote the corresponding envelope by $\mathcal{E}_2.$

\begin{remark} 
The set $\mathcal{E}=\mathcal{E}_1\cup\mathcal{E}_2$ is called the Evolutoid of the surface $S$ at direction $v_{\alpha,\beta}.$
\end{remark}

For the sake of simplicity, throughout this paper, the evolutoid of surface $S$ is called $\mathcal{E}_1$.
%

%
%


\section{Local behavior of evolutoid with singularity view point} \label{local}


We devote this section to a local study on the evolutoid of the surface $S$ using differential geometry techniques. 
Let $f : U \subset \mathbb{R}^n \to\mathbb R^m$ be a smooth map and denote its derivative map by $df : TU \to T\mathbb{R}^m$. The map $f$ is \textit{singular} at $p\in U$ if the rank of the linear map
\[(df)_p : \mathbb{R}^n \to \mathbb{R}^m\]
is not maximal, that is, if $rank(df)_p < min(n,m)$.
The point $p$ is then said
to be a singular point of $f$. Otherwise, we say that $f$ is non-singular at $p$
and $p$ is a \textit{regular} point of $f$.

Two maps $f,g : \mathbb{R}^n \to \mathbb{R}^m$ are said to be $\mathcal{A}$-equivalent if there exist diffeomorphisms $h:\mathbb{R}^n\to\mathbb{R}^n$ and $\xi: \mathbb{R}^m \to \mathbb{R}^m$ such that $f=\xi\circ g\circ h^{-1}$.

The jet space $J^k(n,m)$ is the set of polynomials of degree less than or equal to $k$, without the constant terms. For a given map-germ $f:\mathbb{R}^n\to\mathbb{R}^m$ with $f(0)=0$, one can associate a polynomial in $J^k(n,m)$ to $f$, which is denoted by $j^k f$. Actually, $j^k f$ is simply its Taylor polynomial of degree $k$ at the origin. 

\begin{definition}
	Let $f : (\mathbb R, t_0) \to \mathbb R$ be smooth such that $f^{(i)}(t_0) = 0$ for all $1 \leq i \leq k$, while $f^{(k+1)}(t_0)\ne 0$. Then, we say that $f$ has $A_k$-singularity at $t_0$.
\end{definition}
It is worth mentioning that if $f$ has $A_k$-singularity at $t=t_0$, then it is $\mathcal A$-equivalent to the function $g(t)=\pm t^{k+1}$. A germ-function $g:(\mathbb{R}^m,0) \to \mathbb{R}$ has $A_k$-singularity at the origin if it is $\mathcal A$-equivalent to $\pm x_1^{k+1} \pm x_2^2 \pm \cdots \pm x_m^2.$

Next, we will study the differential geometry of the evolutoid of the surface $S$, considering the singularity theory.


\subsection{Monge form}\label{sec:Monge}


Let $\langle\,,\rangle$ denote the Euclidean inner product. For any regular 
surface $S$ in $\mathbb R^3$ and for any point $q\in S$, we can find an 
isometric change of coordinates $T:(\mathbb R^3,q) \to (\mathbb R^3,0)$
such that $T(S)$ is expressed by the image of the map 
\begin{align}\label{Monge}
\phi : &(\mathbb R^2,0) \to (\mathbb R^3,0)\\
& (u,v) \mapsto (u,v,f(u,v)) \nonumber,
\end{align} 
at least locally, where $$f(u,v)=\dfrac{1}{2}(\kappa_1(0) u^2+\kappa_2(0) v^2) +\sum_{s=3}^{k}c_s+O(k),$$ $c_s=\sum_{i+j=s}\dfrac{a_{ij}}{i!j!} u^iv^j$ and $\kappa_i(0)$ are the principal curvature at the origin for $i=1,2$.
The expression (\ref{Monge}) is called the Monge form pa\-ra\-me\-tri\-za\-tion 
of $S$. Note that the coefficients $\kappa_1$, $\kappa_2$ and $a_{ij}$ are 
differential-geometric invariants of $S$ at $q$. Also, suppose that 
the point $q$ is not neither a parabolic nor  an umbilic point and $\kappa_1(0)>\kappa_2(0)$.  

The Monge form parametrization (\ref{Monge}) can be used to obtain the pa\-ra\-me\-tri\-za\-tion of the envelope (\ref{Envelope}), as demonstrated below.

\begin{proposition}\label{prop:evol}
	Under the above assumptions, the local parametrization of the evolutoid of the surface $S$ at the point $q$ is given by 
	\begin{align}\label{evolutoid}
	\mathcal E_1(u,v) = (h_1(u,v),h_2(u,v),h_3(u,v)),	
	\end{align}
where
\begin{align*}
	h_1(u,v) & = b_{100} + b_{110} u +b_{101} v +O(2),\\
	h_2(u,v) & = b_{220} u^2+b_{202} v^2 +O(3),\\
	h_3(u,v) & = b_{300}+ b_{310} u +b_{301} v +O(2),
\end{align*}
with
$$
\begin{array}{lll}
b_{100} =\dfrac{\cot(\alpha)}{\kappa_2}, & b_{110} =\dfrac{\kappa_2(\kappa_2-\kappa_1)-a_{12}\cot(\alpha)}{\kappa_2^2}, & b_{101}=-\dfrac{a_{03}\cot(\alpha)}{\kappa_2^2}, \vspace{0.3cm}\\
b_{220} = - \dfrac{a_{21}}{2\kappa_2}, & b_{202} = \dfrac{a_{03}}{2\kappa_2},& \vspace{0.3cm}\\
b_{300}  = \dfrac{1}{\kappa_2}, & b_{310}  = \dfrac{-a_{12}+\kappa_1\kappa_2 \cot(\alpha)}{\kappa_2^2}, & b_{301} = - \dfrac{a_{03}}{\kappa_2^2},
\end{array}
$$
here $\kappa_1$ and $\kappa_2$ are the principal curvatures at the point $q$.
\end{proposition}

According to the Proposition \ref{prop:evol}, the evolutoid $\mathcal{E}_1$ passes through 
the point $p=\left(\cot(\alpha)/\kappa_2, 0, 1/\kappa_2 \right)$ at $(0,0)$.

For the surface $S$, there are some special points on $S$ which contain differential geometry information of $S$. Some of these points are ridge points and sub-parabolic points. Geometrically speaking we have:
\begin{definition}
	\begin{enumerate}
		\item A ridge point of $S$ is an $A_3$-singularity of some distance squared function on $S.$ The closure of the set of ridge points is called the ridge of the surface $S.$
		\item The set of points on the
		surface $S$ which correspond to the parabolic set of its focal set is defined in
		\cite{BruceWilkinson,wilkinson} as the sub-parabolic curve of $S$.
	\end{enumerate}
\end{definition}
\begin{remark}
	 It is well-known that the image of the ridge associated to the principal curvature $\kappa_i$ by the
		map $\epsilon_i(p) = p+1/\kappa_i(p)\, N(p)$, for $i = 1,2$, is precisely the singular set of the focal set.
\end{remark}
 For further details on ridge and sub-parabolic points, see \cite{fukui}.
 \begin{proposition}\label{prop:regular}
	Generically, the evolutoid given in \eqref{evolutoid} is regular if and only if the origin is not a ridge point of $S$, with respect to the vector $\partial_v$.
\end{proposition}
\begin{proof}
	A straightforward calculation can demonstrate that the vector $\mathcal{E}_{1_u}\times\mathcal{E}_{1_v}$ is not zero at $p$ if 
$(\sin^2(\alpha) \kappa_2(0)-\kappa_1(0))a_{03}\ne 0.$ We assume that $\kappa_1(0)-\kappa_2(0)>0$ so $\sin^2(\alpha) \kappa_2(0)-\kappa_1(0)$ is different from zero. Therefore, the normal vector of the evolutoid at $p$ is not zero if and only if $a_{03}\ne 0$, which is equivalent to say that the origin is not a ridge point of $S$ with respect to the vector $\partial_v$.
\end{proof}

\subsection{Contact with plane}

After performing a suitable change of coordinate on \eqref{evolutoid}, 
let $\psi : U \to \mathbb{R}^3$ be the Monge form parametrization of the regular evolutoid $\mathcal{E}_1$ of the surface $S$. The contact of $\mathcal{E}_1$
with planes is measured by the singularities of the height functions on $\mathcal{E}_1$.

Recall that the family of height functions $H : U\times S^2 \to \mathbb{R}$ on $\mathcal{E}_1$ is given by
\[H(u,v,{\bf v})=\langle \psi(u,v),\bf v\rangle,\]
where ${\bf v}$ is a unit vector in $\mathbb R^3$.
 The contact of a regular surface with planes is well studied (see for instance \cite{izumiya}). In this subsection, we shall derive geometric information about the original surface $S$ from the family of height functions of its evolutoid. 
 Before stating our results, it is worth mentioning the following Proposition about the singularities of the height function of the evolutoid. In \cite{izumiya} is presented the singularities of height function for any surface.
 \begin{proposition}[\cite{izumiya}]
 	The height functions $h_{\bf v}(u,v)=H(u,v,{\bf v})$ is singular at $p\in \mathcal{E}_1$ if and only if ${\bf v}$ is a normal direction to $\mathcal{E}_1$ at $p$ $($which is denoted by $N_p)$. The singularity of $h_{\bf v}$ at $p$ is of type $A_1^+$ if and only if $p$ an elliptic point, $A_1^-$ if and only if $p$ a hyperbolic point, and $A_2$ if and only if $p$ is a parabolic point.
 \end{proposition}
\begin{proof}
Let $u_i,u_j\in\{u,v\}$ for $i,j=1,2$. We have 
$$\dfrac{\partial h_{\bf v}}{\partial u_i}(u,v)=\langle\psi_{u_i}(u,v),{\bf v}\rangle=0\Leftrightarrow {\bf v}=N_p.$$
Now fix ${\bf v}=N_p$ we have:
\begin{align*}
\dfrac{\partial^2h_{\bf v}}{\partial u_i\partial u_j}(u,v)=\langle\psi_{u_iu_j}(u,v),N_p\rangle = \Gamma_i^j, 
\end{align*}  
where $\Gamma_i^j$ are the coefficients of the second fundamental form of $\mathcal{E}_1$. Since the sign of the Hessian matrix $(\Gamma_i^j)_{2\times 2}$ and the Gauss-Kronecker curvature of $\mathcal{E}_1$, denoted by $\tilde{K}$, at $p = \psi(u,v)$ are the same, then the type of the singularity of $h_{\bf v}$ is distinguished by the sign of the determinant of the Hessian matrix. This means that
\begin{align*}
h_{\bf v} \, \text{has an} \, A_1^+\text{-singularity at}\, p &\Leftrightarrow \, p\,\text{is an elliptic point},\\
h_{\bf v} \, \text{has an} \, A_1^-\text{-singularity at}\, p &\Leftrightarrow \, p\,\text{is an hyperbolic point}.
\end{align*}
The curvature $\tilde{K}$ at $p = \psi(u,v)$ is zero if and only if the determinant of the Hessian matrix $(\Gamma_i^j)_{2\times 2}$, is zero. This means that $p$ is a parabolic point if and only if $h_{\bf v}$ has $A_2$-singularity.  
\end{proof} 
Some important differential geometric information of the evolutoid are determined by our next theorems.
 
 \begin{theorem}\label{Prop:assymp}
 	Under the assumptions stated above, the following results hold for $\mathcal{E}_1$ given in \eqref{evolutoid}, e\-qui\-pped with the Monge form parametrization $\psi$ at $p=\psi(0,0)$ with ${\bf v}=\psi_y(0,0)\in T_p\mathcal{E}_1$.
 	\begin{itemize}
 		\item[i)] The point $p$ is a parabolic point of the evolutoid $\mathcal{E}_1$ if and only if the origin is a sub-parabolic point of $S$.
 		\item[ii)] The direction ${\bf v}$ is asymptotic at $p$ if and only if $$a_{21} =\frac{\left(a_{12} - \kappa_2(0) (\kappa_2(0)-\kappa_1(0)) \tan(\alpha)\right)^2}{a_{03}}, $$
 		where $\kappa_1(0)$ and $\kappa_2(0)$ are the principal curvatures of $S$ at the origin.
 		\item[iii)] Suppose that the origin is not a parabolic point of $S$. If ${\bf v}$ is an asymptotic direction at $p$ and $p$ is a parabolic point of $\mathcal{E}_1$, then the parabolic set of $\mathcal{E}_1$ is a smooth curve at $p$, if and only if $a_{31}\ne 0$.
 	\end{itemize}
 \end{theorem} 
\begin{proof}
	For the evolutoid $\mathcal{E}_1$, let $\tilde{l},\tilde{m}$ and $\tilde{n}$ denote the coefficients of the second fundamental form. We have $\tilde{l}(0,0)=-(a_{21}/\kappa_2(0))$, $\tilde{m}(0,0)=0$ and $\tilde{n}(0,0)=(a_{03}/\kappa_2(0))$. 
	
	i) The point $p$ is a parabolic point of the evolutoid $\mathcal{E}_1$ if and only if the numerator of the parabolic curve $\tilde{K}$ at $(0,0)$ is zero. Therefore, we have \[(\tilde{l}\tilde{n}-\tilde{m}^2)(0,0)=-\dfrac{\kappa_2^4 a_{21}}{a_{03} \left(-\kappa_2(0) + \kappa_1(0) + \kappa_1(0) \cot^2(\alpha)\right)^2},\] which is zero if $a_{21}=0$. This means that the origin is a sub-parabolic point of the surface $S$ with respect to the vector $\partial_u$.
	
	ii) The direction ${\bf v}$ is asymptotic if $\tilde{\kappa}_n({\bf v})=0$, where $\tilde{\kappa}_n$ denotes the normal curvature of $\mathcal{E}_1$. Therefore, the vector ${\bf v} = (0, 0, 1)$ is an asymptotic direction at
	$p$ if and only if 
	\begin{align}\label{assym}
	a_{02}^2 (\kappa_2(0)-\kappa_1(0))^2 - 
	2 \kappa_2(0) a_{12} (\kappa_2(0)-\kappa_1(0)) \cot(\alpha) +&\nonumber\\ +(a_{12}^2 - a_{03} a_{21}) \cot^2(\alpha) & = 0.
	\end{align}
	As $a_{03}\ne 0$, we obtain $$a_{21} =\dfrac{\left(a_{12} - \kappa_2(0) (\kappa_2(0)-\kappa_1(0)) \tan(\alpha)\right)^2}{a_{03}}$$ from (\ref{assym}).
	
	iii) Let $\tilde{K}(u,v)=\tilde{K}$ denote the parabolic curve of the evolutoid. When {\bf v} is an asymptotic direction and $p$ is a parabolic point, by using (i) and (ii), we have:  
	\begin{align*}
	\tilde{K}= \frac{\kappa_2^5(0) a_{31} \tan^2(\alpha)}{a_{03} (\kappa_1(0) \cot(\alpha) + (-\kappa_2(0) + \kappa_1(0)) \tan^3(\alpha)}(\tan(\alpha) u-v)+O(2).
	\end{align*}
	As the origin is not a parabolic or umbilic point of $S$, $\kappa_2(0)\ne 0$. Hence, $\tilde{K}$ is smooth at $p$, if and only if $a_{31}\ne 0$.  
\end{proof}
\begin{theorem}\label{Theo:height}
	For a regular surface $S$ parametrized in Monge form \eqref{Monge}, let $\mathcal{E}_1$ denote its evolutoid given in \eqref{evolutoid}. Also, consider the point $p=\psi(0,0)$ and ${\bf v}=(0,0,1)\in T_p\mathcal{E}_1$ where $\psi : (\mathbb{R}^2,0) \to (\mathbb{R}^3,p) $ denotes the Monge form parametrization of $\mathcal{E}_1$. Assume that the origin is not a flat umbilic point of $S$. Then for $p$ a parabolic point of $\mathcal{E}_1$ and ${\bf v}$ an asymptotic direction at $p$, the height function $h_{N_0}=H(u,v,N_0)$ has a singularity at $p$ of type
	\begin{align*}
	& A_2 \Leftrightarrow a_{31}\ne 0, \\
	& A_3 \Leftrightarrow a_{31}=0 \,\, \text{and}\,\, a_{41}\ne 0.
	\end{align*}
\end{theorem}
	\begin{proof}
	After performing a suitable change of coordinates on the evolutoid (\ref{evolutoid}), one can assume that $\psi(u,v)=(u,g(u,v),v)$ with $j^1g=0$. The height function along the normal direction $N_0=(0,1,0)$ at the point $p$ is given by $h_{N_0}(u,v)=g(u,v)$. According to the Proposition \ref{prop:regular} and Theorem \ref{Prop:assymp}, if $\mathcal{E}_1$ is regular, $p$ is a parabolic point of $\mathcal{E}_1$ and {\bf v} is an asymptotic direction at $p$, then $a_{03}\ne 0$, $a_{21}=0$ and $a_{12}=\kappa_2(0) (\kappa_2(0)-\kappa_1(0)) \tan(\alpha)$. The coefficient of $u^2$ in $g(u,v)$ is equal to $(\kappa_2^3(0) \tan^2(\alpha))/a_{03}$, which is not zero since $\kappa_2(0)\ne 0$. Using some appropriate change of coordinates, one can eliminate the term $uv$ in the height function $g$ and conclude that $h_{N_0}$ has an $A_2$-singularity if and only if $a_{31}\ne 0$. 
	
 The height function has an $A_{\geq 2}$-singularity at the origin if and only if $a_{31}=0$. Then we have $j^4g(u,v)=g_{20} u^2+g_{04} v^4,$ where
 \begin{align*}
 g_{20} = & \dfrac{\kappa_2^3(0)\tan^2(\alpha)}{2a_{03}}, \\
 g_{04} = & - \dfrac{\kappa_2^3(0) a_{41}}{24\left(\kappa_1(0)\cot(\alpha)+(\kappa_1(0)-\kappa_2(0))\tan(\alpha)\right)^4}.
 \end{align*}   
	Clearly, the height function has an $A_3$-singularity at the origin if and only if  $a_{41}\ne 0$. Note that 
\[\kappa_1(0)\cot(\alpha)+(\kappa_1(0)-\kappa_2(0))\tan(\alpha)\ne 0\]
because if it is zero then 
\begin{equation*}
\begin{array}{rccc}
\dfrac{\kappa_1(0)}{\tan(\alpha)}+(\kappa_1(0)-\kappa_2(0)\tan(\alpha)&= &0 & \Leftrightarrow \\
 \dfrac{1}{\cos^2(\alpha)}\kappa_1(0) - \dfrac{\sin^2(\alpha)}{\cos^2(\alpha)}\kappa_2(0) &=&0& \Leftrightarrow \\
 \kappa_1(0) - \sin^2(\alpha) \kappa_2(0)& =&0,&
 \end{array}
\end{equation*}
which contradicts our assumption $\kappa_1(0)-\kappa_2(0) > 0$.
    
\end{proof}
A point on the parabolic set where the height function
along the normal direction to the surface has an $A_3$-singularity is called a \textit{cusp of Gauss}. A cusp of Gauss point $q$ of a regular surface in $\mathbb{R}^3$ is called elliptic (resp. hyperbolic) if the height function along the normal direction at $q$ has an $A^+_3$-singularity (resp. $A^-_3$-singularity) at $q$ (for further details see \cite{izumiya}).
\begin{remark}
	 Suppose that the point $p=\psi(0,0)$ is the cusp of Gauss point of the evolutoid $\mathcal{E}_1$. By Theorem \ref{Theo:height}, one can conclude that $p$ is elliptic if $a_{41}>0$ or hyperbolic, if $a_{41}<0$.
\end{remark}
\begin{proposition}
	Let $S$ and its evolutoid $\mathcal{E}_1$ be given in Monge forms $\phi$ and $\psi$ at $(0,0)$ as in \eqref{Monge} and $(\ref{evolutoid}),$ respectively.
	Also let ${\bf v} = (0,0,1) \in T_p\mathcal{E}_1$, where $p=\psi(0,0)$.
	If {\bf v} is an asymptotic direction at $p$ and $p$ is a parabolic point but not a cusp of Gauss, then the geodesic curvature of $\mathcal{E}_1$ at $N(p)$ of the image of the parabolic curve by the Gauss map is
	\[ \tilde{\kappa}_g =\frac{a_{03} \cos(\alpha) \cot^3(\alpha) (-a_{13} \cos(\alpha) + a_{03} (\kappa_2(0)-\kappa_1(0)) \sin(\alpha))}{\kappa_2^5(0) (-\kappa_2(0) + 2 \kappa_1(0) + \kappa_2(0) \cos(2\alpha))}. \]
\end{proposition} 
\begin{proof}
 The proof follows by standard techniques and definitions of differential geometry and is omitted.
\end{proof}
\subsection{Singular evolutoid}
According to Proposition \ref{prop:regular}, the evolutoid $\mathcal{E}_1$ given in \eqref{evolutoid} is singular at $(0,0)$ if $a_{03}=0.$ We require the following definition in order to further our understanding of the singular points of the evolutoid surface.
\begin{definition}\label{def:sing}
Let $U$, $\tilde{U}$ be open subsets of $\mathbb R^2$ and $V$, $\tilde{V}$ be open subsets of $\mathbb R^3$. A surface $\psi : U\to V$ has a cuspidal edge or swallowtail or cross cap singularity if $\psi$ be $\mathcal{A}$-equivalent to the map ${\bf x}: \tilde{U}\to\tilde{V}$ given in the following forms (see Figure \ref{fig:CuspSwall}):
\begin{itemize}
	\item Cross cap: $(t,s)\mapsto (t, ts, s^2)$;
	\item Cuspidal-edge: $(t,s)\mapsto(t,s^2,s^3)$;
	\item Swallowtail: $(t,s)\mapsto(t,2st+4s^3,ts^2+3s^4)$.
\end{itemize}	
\end{definition}
\begin{figure}[!htb]
	\includegraphics[clip, scale=0.7]{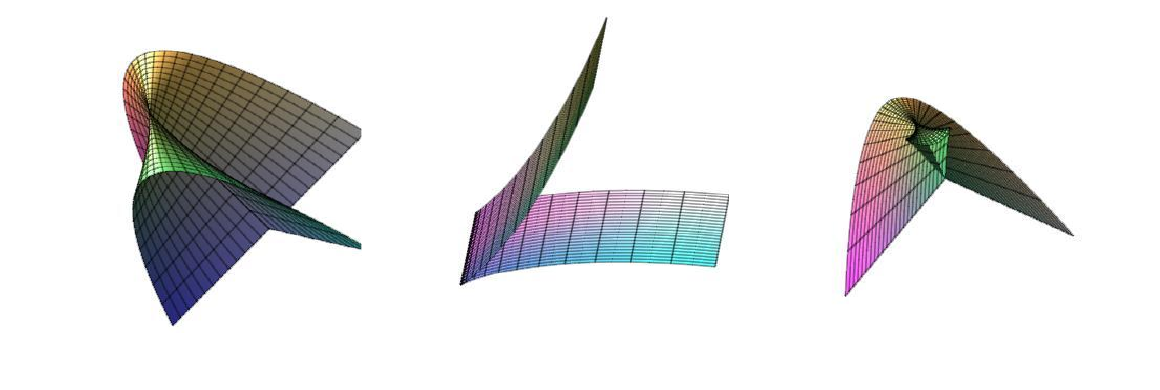}
	\caption{Cross cap singularity (left), cuspidal-edge singularity (center) and Swallowtail singularity (right).}
	\label{fig:CuspSwall}
\end{figure}
We begin with a general results on criterion of these singularities given in 
Definition \ref{def:sing}. The following is well-known as Whitney criterion for cross cap.
\begin{proposition}[Whitney]\label{prop:cross cap}
	Assume a surface $P(u,v)$ has a singular point at $(u_0,v_0)$ and $P_u(u_0,v_0)=0$. Then $(u_0,v_0)$ is a cross cap if and only if 
	\[\det\left(P_v(u_0,v_0),P_{uu}(u_0,v_0),P_{uv}(u_0,v_0)\right)\neq 0.\]
\end{proposition}
Note that for a cross cap singularity, by a rotation of the $uv$-coordinate plane, one can choose a coordinate system such as $P_u(u_0,v_0)=0$.

As cuspidal-edge and swallowtail singularities are not finitely determined (see \cite{izumiya} for more details) so the criterion to distinguish these singularities are more complicated and technical. We use \cite{MaKo} to prepare some notations and general ideas: The gradient vector field of a function $f(u,v)$ is denoted by $\nabla f:=(f_u,f_v)$. Moreover, for a planar vector $w=(\theta_1,\theta_2)$, the function $f_w:=\theta_1 f_u+\theta_2 f_v =\langle \nabla f, w\rangle$ is the directional derivative of $f$ with respect to $w$. Similarly, for a surface $P(u,v)$ we set $P_w:=\theta_1 P_u+\theta_2 P_v$. 
\begin{remark}
	Some common properties of cuspidal-edges and swallowtails are as follows:
	\begin{itemize}\label{rem:CuspSwal}
		\item[(i)] The unit normal vector field $\nu$, can be extended smoothly to the singular point.
		\item[(ii)] The function $\lambda:=\det(P_u,P_v,\nu)$ vanishes on the singular point but $\nabla\lambda$ does not vanish at the singular points.
		\item[(iii)] The singular set of the surface $P: U\subset\mathbb{R}^2\to \mathbb{R}^3$ is expressed as 
		\[\{(u,v)\in U\,|\,\lambda(u,v)=0 \}.\]
		Since $\nabla\lambda\neq 0$ at the singular point $(u_0,v_0)$ then the implicit function theorem yields that the singular set is a regular curve on a neighbourhood of $(u_0,v_0)$ in the $uv$-plane. This curve is called the singular curve and denoted by $\gamma(t)=(u(t),v(t))$.
		\item[(iv)] At each point on $\gamma(t)$, $P_u$ and $P_v$ are linearly dependent and do not vanish simultaneously. Therefore, for each $t$, there exists a vector $\eta(t)=(\theta_1(t),\theta_2(t))\ne (0,0)$ uniquely determined up to scalar multiplication such that
		\[P_{\eta(t)}=\theta_1(t) P_u(\gamma(t))+\theta_2(t) P_v(\gamma(t))=0.\]
		The vector $\eta(t)$ is called null vector at the singular point $\gamma(t).$
		\item[(v)] Using notations stated above, we have $\nu_{\eta(t)}\neq 0$.
		In this situation we set
		\[\mu(t):=\det\left(\dot{\gamma}(t),\eta(t)\right)=\dot{u}(t)\theta_2(t)-\dot{v}(t) \theta_1(t),\]
		where ``$\,\,^.$" denotes the derivative with respect to $t$.
	\end{itemize}
\end{remark}
\begin{proposition}[Criteria for cuspidal-edge and swallowtail \cite{MaKo}]
	Let the surface $P(u,v)$ has a singular point at $(u_0,v_0)$ and satisfies the properties given in Remark \ref{rem:CuspSwal}. Then the point $\gamma(t_0)=(u_0,v_0)$ corresponds to:
	\begin{itemize}
		\item[(1)]  a cuspidal-edge if and only if $\mu(t_0)\ne 0$. This means that, $\dot{\gamma}$ and $\eta$ are linearly independent at $t=t_0$.
		\item[(2)] a swallowtail if and only if 
		\[\mu(t_0)=0 \quad \text{and}\quad \dot{\mu}(t_0)\ne 0.\]
	\end{itemize}  
\end{proposition}
Now we are ready to investigate the singularity types of the evolutoid of a regular surface.
\begin{theorem}\label{Theo:Sing}
	Let  $S$ be a regular surface in $\mathbb R^3$ equipped with the Monge form parametrization $\phi$ at the origin as in \eqref{Monge}. For a given $0<\alpha<\pi/2$, let $\mathcal{E}_1$ be the evolutoid of $S$, as in \eqref{evolutoid}. If the $p=\mathcal{E}_1(0,0)$ is a singular point then it has a singularity of type cuspidal-edge.
	\end{theorem} 
\begin{proof}
	First of all note that, due to Proposition \ref{prop:regular} the point $p\in\mathcal{E}_1$ is singular if and only if $a_{03}=0$. For the sake of simplicity we denote the parametrization of the evolutoid by $\psi$. We have:
	
	(i) Due to Proposition \ref{prop:cross cap} the point $p$ has a cross cap singularity at $(0,0)$ if and only if $\det(\psi_v,\psi_{uu},\psi_{uv})\ne 0$ at $(0,0)$. But we have
	\begin{align*}
	\det(\psi_v,\psi_{uu},\psi_{uv}) = \det\left[\begin{array}[pos]{ccc}
	C_{11} & 0 & C_{13} \\
	C_{21} & 0 & C_{23}\\
	C_{31} & 0 & C_{33}
	\end{array}\right] = 0,
	\end{align*}
	where $C_{ij}$ with $i=1,2,3$ and $j=1,3$ are generally non null.
This implies that $p$ does not have cross cap singularity.

(ii) The unit normal vector $\nu$ has the following expression:
\begin{align*}
\nu=\frac{\psi_u\times\psi_v}{||\psi_u\times\psi_v||}=(\nu_1,\nu_2,\nu_3),
\end{align*}
where
\begin{align}
\nu_1= & \dfrac{a_{21} u + (a_{12}  - \kappa_1 \kappa_2  \cot(\alpha))v}{-\kappa_2 + \kappa_1 + \kappa_1 \cot(\alpha)^2} + O(2), \nonumber \\ 
\nu_2 = & -1+O(2), \label{nu} \\
\nu_3= & \dfrac{(\kappa_2 (\kappa_2 - \kappa_1) v - (a_{21} u + a_{12} v) \cot(\alpha)}{-\kappa_2 + \kappa_1 + \kappa_1 \cot(\alpha)^2}+O(2),\nonumber
\end{align}
where $\kappa_1$ and $\kappa_2$ are the principal curvatures of $S$ at the origin.

Therefore, the function $\lambda(u,v)=\det(\psi_u,\psi_v,\nu)= \lambda_{10} u + \lambda_{01} v + O(2)$ where $\nabla\lambda\ne (0,0)$. Hence, we can write $v$ as a function of $u$ and we have the singular curve $\gamma(t)=(t,v(t))$ where $v(t)=-\lambda_{10}/\lambda_{01} t +O(2)$. As $\psi_u$ and $\psi_v$ are linear dependent and do not vanish simultaneously. Therefore, at $t=0$ we have
\begin{align*}
\psi_u(0,0) & =\left(\dfrac{\kappa_2^2-\kappa_2 \kappa_1-a_{12}\cot(\alpha)}{\kappa_2^2}, 0, \dfrac{-a_{12}+\kappa_2 \kappa_1\cot(\alpha)}{\kappa_2^2}\right)\\
\psi_v(0,0) & =(0,0,0),
\end{align*} 
and $\tan(\alpha)^2 \ne \kappa_1/(\kappa_2-\kappa_1)$.
Thereupon, the vector $\eta(0) = (0,1)$ is the null vector at $\gamma(0)=(0,0)$. In this situation we have 
\begin{equation*}
\mu(0)=\det(\dot{\gamma}(0),\eta(0))=\det\left(\begin{matrix}
1 & -\dfrac{\lambda_{10}}{\lambda_{01}}\\
0 & 1
\end{matrix}\right) =1 
\end{equation*} 
Hence according to the Remark \ref{rem:CuspSwal}, the point $\gamma(0)=(0,0)$ corresponds to cuspidal edge singularity. Obviously, $\gamma(0)=(0,0)$ is not a swallowtail singularity.
\end{proof}



\bibliographystyle{amsplain}
\bibliography{references}

\end{document}